\documentclass[reqno,a4paper,12pt]{amsart}
\usepackage{amssymb,amsmath,amscd,amstext,amsthm,amsfonts}
\usepackage[]{graphicx}
\usepackage{epstopdf}
\usepackage{setspace} 
\usepackage{subfig}
\usepackage{color}
\usepackage[mathscr]{eucal}
\usepackage[draft]{changes}
\usepackage[ansinew]{inputenc} 
\usepackage{hyperref}
\usepackage{color}

\numberwithin{equation}{section}

\newtheorem{theorem}{Theorem}[section]
\newtheorem{proposition}[theorem]{Proposition}

\newtheorem{lemma}[theorem]{Lemma}

{\theoremstyle{definition}
{
\newtheorem{remark}[theorem]{Remark}
\newtheorem{example}[theorem]{Example}

\newtheorem{defn}[theorem]{Definition}
}}

\newcommand{\cal}{\mathcal}

\newcommand{\BB}{{\cal B}}

\newcommand{\II}{{\cal I}}

\newcommand{\PP}{{\cal P}}

\newcommand{\VV}{{\cal V}}
\newcommand{\WW}{{\cal W}}
\newcommand{\XX}{{\cal X}}
\newcommand{\YY}{{\cal Y}}
\newcommand{\ZZ}{{\cal Z}}
\newcommand{\SRB}{\mathscr{E}}

\newcommand{\Nn}{{\mathbb{N}}}

\newcommand{\Rr}{{\mathbb{R}}}

\DeclareMathOperator{\sint}{int}

\def\dist{\operatorname{dist}}

\def\supp{\operatorname{supp}}

\def\ord{\operatorname{ord}}

\def\Per{\operatorname{Per}}

\newcommand{\abs}[1]{\left| #1\right|}

\newcommand{\id}  {\operatorname{id}}

\newcommand{\comment}[1]{}

\newcommand{\Sone}{[0,1]}

\begin{document}
\title[Attractor of piecewise expanding maps]{On the attractor of piecewise expanding maps of the interval}

\date{\today}

\author[Del Magno]{Gianluigi Del Magno}
\address{Universidade Federal da Bahia, Instituto de Matem\'atica\\
Avenida Adhemar de Barros, Ondina \\
40170110 - Salvador, BA - Brasil}
\email{gdelmagno@ufba.br}

\author[Lopes Dias]{Jo\~ao Lopes Dias}
\address{Departamento de Matem\'atica and CEMAPRE, ISEG\\
Universidade de Lisboa\\
Rua do Quelhas 6, 1200-781 Lisboa, Portugal}
\email{jldias@iseg.ulisboa.pt}

\author[Duarte]{Pedro Duarte}
\address{Departamento de Matem\'atica and CMAF \\
Faculdade de Ci\^encias\\
Universidade de Lisboa\\
Campo Grande, Edificio C6, Piso 2\\
1749-016 Lisboa, Portugal 
}
\email{pmduarte@fc.ul.pt}

\author[Gaiv\~ao]{Jos\'e Pedro Gaiv\~ao}
\address{Departamento de Matem\'atica and CEMAPRE, ISEG\\
Universidade de Lisboa\\
Rua do Quelhas 6, 1200-781 Lisboa, Portugal}
\email{jpgaivao@iseg.ulisboa.pt}

\begin{abstract}
We consider piecewise expanding maps of the interval with finitely many branches of monotonicity and show that they are generically combinatorially stable, i.e., the number of ergodic attractors and their corresponding mixing periods do not change under small perturbations of the map. Our methods provide a topological description of the attractor and, in particular, give an elementary proof of the density of periodic orbits. 
\end{abstract}

\maketitle

\section{Introduction}
\label{sec:introduction}
In this paper we study the dynamics of piecewise expanding maps of the interval. We fix $m\in \Nn$. A map  $f \colon \Sone\to \Sone$ is called {\it piecewise expanding} on $m$ intervals if there exist a constant $\sigma>1$ and intervals $I_1,\ldots, I_m$  such that  
\begin{enumerate}
\item $\Sone=\bigcup_{i=1}^m I_i$ and ${\rm int}(I_i)\cap {\rm int}(I_j)=\emptyset$ for $i\neq j$,
\item $f$ is $ C^{1} $ and $ | f'| \ge \sigma$ on each $ I_i $,
\item $f'$ is Lipschitz on each $ I_i $\footnote{This implies that $f'|_{\sint(I_i)}$ has a continuous extension to the closure of $I_i$.}.
\end{enumerate}

The dynamics of this class of maps has been widely studied as it finds applications in other areas of mathematics and in many other branches of science~\cite{BG97}. The theory of piecewise expanding maps is by now rather satisfactory from a probabilistic point of view.  Computer simulations show that typical orbits of piecewise expanding maps display chaotic behaviour as they approach an attractor. A way to describe the chaotic behaviour on the attractor is through the study of invariant measures~\cite{Young02,Viana}. It is well-known that piecewise expanding maps admit absolutely continuous (with respect to Lebesgue measure) invariant probability measures, known as acip's~\cite{L-Y}. They are physically meaningful since it allows us to understand the statistical behaviour of positive Lebesgue measure sets of orbits. 

Deterministic and random perturbations of piecewise expanding maps have been considered by many people, e.g.,~\cite{K82,G84,B89,BY94,K-L,DL,eg13}. A key concept is stability. Roughly speaking, a map is called stable if its statistical properties are robust under small perturbations of the map. In the context of piecewise expanding maps, a map $f$ with an acip $\mu$ is acip-stable if given any small perturbation $f_\varepsilon$ of $f$ we have that $\mu_\varepsilon$ converges to $\mu$ in the weak*-sense as $\varepsilon\to0$ where $\mu_\varepsilon$ is an acip of $f_\varepsilon$~\cite{eg13}.



%
%

%
%
%

In this paper we are interested in determining which piecewise expanding maps have robust combinatorics at the level of the attractor. To be more precise, we say that a piecewise expanding map $f$ is \textit{combinatorially stable} if the number of ergodic acip's of
$ f $ and the corresponding mixing periods do not change in a neighbourhood of $f$. 

The following theorem is the main result of this paper. 

\begin{theorem}\label{main thm}
Generic piecewise expanding maps on $ m $ intervals are combinatorially stable and the supports of their acips vary continuously with the map. 
\end{theorem}

The sufficient conditions on the maps for the main theorem to hold are given in Definition~\ref{def:H}. They are generic by Proposition~\ref{prop:generic}. In Section~\ref{sec stability} we give a proof and a precise formulation of Theorem~\ref{main thm} (see Theorem~\ref{thm:main2}). 



In addition, we prove several results for piecewise expanding maps used in the proof of Theorem~\ref{main thm}.  These results may be of independent interest. For example, in Section~\ref{sec top properties}, using elementary methods we prove the following. 

\begin{theorem}
The periodic points of any piecewise expanding map are dense in the support of the acips.
\end{theorem}  

The density of periodic points might not be surprising, nevertheless there are no references in the literature as far we are aware. 
In addition to determining the number of ergodic 
components, the method for proving Theorem~\ref{main thm} provides a topological description and the continuity of the immediate basins of attraction, which complements the results obtained by a spectral approach~\cite{K82,K-L}.

The strategy for proving Theorem~\ref{main thm} is the following. To any generic piecewise expanding map $f$ and ergodic acip $\mu$ of $f$  we associate a trapping region $U_\mu(g)$ for small perturbations $g$ of $f$. This trapping region contains the support of an ergodic acip $\mu_g$ of the perturbed map $g$. Using the density of periodic points in the support of the acips we show that $\mu_g$ is unique, i.e., no other ergodic acip  of $g$ has its support inside $U_\mu(g)$. So we have a well defined map $\Theta_g\colon\mu\mapsto\mu_g$ from the set of ergodic acips of $f$ to the set of ergodic acips of $g$. Then we prove that $\Theta_g$ is a bijection. This shows that $f$ and $g$ have the same number of ergodic acips. Working in a similar way, we conclude that $f$ and $g$ have the same number of mixing components.

We believe that the proof of Theorem~\ref{main thm} might be adapted to cover two-dimensional hyperbolic maps with singularities which are close, in an appropriate sense, to a one-dimensional piecewise expanding map. A special class of two-dimensional hyperbolic maps with singularities are the strongly dissipative polygonal billiards~\cite{lxds}. The combinatorial stability for this class of dissipative billiards will be treated in a separate paper.
 

The rest of the paper is organized as follows. In Section~\ref{sec preliminaries} we introduce some notation and recall a well-known theorem concerning the existence of acips for piecewise expanding maps. Several topological properties of the attractor are proved in Section~\ref{sec top properties}. In Section~\ref{sec stability}, we prove our main result regarding the combinatorial stability of piecewise expanding maps.

\section{Preliminaries}\label{sec preliminaries}
Let $f$ be a piecewise expanding map. Throughout the paper, we use the standard abbreviation \emph{acip} for an invariant probability measure of $f$ that is absolutely continuous with respect to the Lebesgue measure of $ \Sone $. We also write `($ \bmod \; 0 $)' to specify that an equality holds up to a set of zero Lebesgue measure. The length of an interval $I\subseteq [0,1]$ is denoted by $|I|$. Given any subset $A\subset [0,1]$, its boundary $\partial A$ and interior $\sint (A)$ are taken relative to $\Rr$.

\subsection{Existence of acips}

Given a Borel measure $ \mu $, we denote by $ \supp\mu $ the smallest closed set of full $ \mu $-measure. 

We say that the pair $(f,\mu)$, where $\mu$ is an acip of $f$, is \textit{exact}\footnote{By \cite[Theorem 3.4.3]{BG97}, $(f,\mu)$ is exact if and only if for any Borel set $B\subset[0,1]$ with $\mu(B)>0$, $\lim_{n\to\infty}\mu(f^n(B))=1$.} if
$$
\bigcap_{n=0}^\infty f^{-n}(\BB)
$$
consists of $\mu$-null sets and its complements. Here, $\BB$ denotes the Borel $\sigma$-algebra.

\begin{theorem}
\label{th:acip}
Let $f$ be a piecewise expanding map.
Then, 
\begin{enumerate}
\item  there exists $ 1 \le k \le m $ such that
$ f $ has exactly 
$k$ ergodic acip's $ \mu_{1},\ldots,\mu_{k} $ with bounded variation densities,
\item for every $ 1 \le i \le k $, there exist $k_i \in \Nn $ and 
an acip $\nu_i$ such that
\begin{enumerate}
\item
$$
\mu_i=\frac1{k_i}\sum_{j=0}^{k_i-1}f_*^{j}\nu_i
$$
\item
$(f^{k_i}, f_*^j\nu_i)$ is exact  for all $j$,
\end{enumerate}
\item $ \supp f^{j}_{*} \nu_{i} $ and $ \supp\mu_{i} $ are both unions of finitely many pairwise disjoint intervals,   for all $j$,
\item the union of the basins of $ \mu_{1},\ldots,\mu_{k} $ is equal $ (\bmod \; 0) $ to $ \Sone $.
\end{enumerate}
\end{theorem}

\begin{proof}
Notice that, by \cite[pp. 17 and Theorem 2.3.3]{BG97}, $1/|f'|$ has bounded variation. Thus, 
Parts~(1) and (2) are proved in~\cite[Theorems 7.2.1 and 8.2.1]
{BG97}. It suffices to prove Part~(3) for $ \supp f^{j}_{*} \nu_{i} $ which we do applying~\cite[Theorem~8.2.2]{BG97} to the acip $ f^{j}_{*} \nu_{i} $ of $ f^{k_{i}} $. 
Part~(4) is proved in \cite[Theorem 3.1, Corollary 3.14]{Viana}. 
\end{proof}

Let
\[ 
\Lambda_{i,j} := \supp f_*^j\nu_i
\]  
We call $ \Lambda_{i,1},\ldots,\Lambda_{i,k_{i}} $ and $ k_{i} $ in Theorem~\ref{th:acip} the \textit{mixing}\footnote{In view of Part~(2) of Theorem~\ref{th:acip}, one may be tempted to call these components exact rather than mixing. However, mixing and exactness are equivalent concepts for a piecewise expanding map~\cite[Theorem~7.2.1]{BG97}.} \textit{components} and the \textit{mixing period} of $ \mu_{i} $, respectively. We also define $\Per(\mu_i):=k_i$.

\subsection{Topology}

Now we introduce a topology on the space of piecewise expanding maps. Recall that a piecewise expanding map $f$ is defined by a partition $\PP_f=\{I_1,\ldots,I_m\}$ of the interval $[0,1]$ with boundary points 
\begin{equation}\label{partition}
0=a_0<a_1<\cdots<a_{m-1}<a_m=1.
\end{equation}
To stress the dependence of $I_i$ and $a_i$ on $f$ we shall write $I_i(f)$ and $a_i(f)$, respectively.

Denote by $\XX_m$ the set of piecewise expanding maps on $m$ intervals. Given $f,g\in\XX_m$ define

$$
d(f,g):=\rho(\PP_f,\PP_g)+\rho_0(f,g)+\rho_{\text{Lip}}(f',g'),
$$
where
\begin{align*}
\rho(\PP_f,\PP_g)&:=\max_{1\leq i\leq m-1}|a_i(f)-a_i(g)|\\
\rho_0(f,g)&:=\max_{1\leq i\leq m} \|f-g\circ\eta_i\|_{C^0(I_i(f))} \\
\rho_{\text{Lip}}(f',g')&:=\max_{1\leq i\leq m}\left\{\|f'-g'\circ\eta_i\|_{\text{Lip}(I_i(f))}+\|g'-f'\circ\eta_i^{-1}\|_{\text{Lip}(I_i(g))}\right\}
\end{align*}
and $\eta_i:I_i(f)\to I_i(g)$ is the affine function that maps $a_i(f)$ to $a_i(g)$. Here, $\|\cdot\|_{C^0}$ and $\|\cdot\|_{\text{Lip}}$ denote the usual norm of a continuous and Lipschitz function, respectively. Clearly, $d$ is a metric, thus $(\XX_m,d)$ is a metric space. In fact, $(\XX_m,d)$ is a complete metric space.

In this paper, a neighbourhood $\VV$ of $f$ is always to be understood in the metric $d$. Notice that, given any sequence of piecewise expanding maps $f_n\in\XX_m$ converging to $f\in \XX_m$, $f_n$ also converges to $f$ in the Skorokhod-like metric (cf.~\cite{eg13}).

\section{Topological properties}\label{sec top properties}

Let $f$ denote a piecewise expanding map. In this section we derive several topological properties of the attractor of $f$. Some of these properties will be used to prove the combinatorial stability of the attractor (see Theorem~\ref{thm:main2}).

\subsection{Boundary segments}

Let
\begin{equation}
\label{def:set D}
D=D_f:=\{a_1,\ldots,a_{m-1}\},
\end{equation}
where $a_i$ are the points in \eqref{partition} forming the partition $\PP_f$ of $f$. We denote by $D_{f^n}$ the set of points $x\in[0,1]$ such that $f^k(x)\in D$ for some $0\leq k<n$.
Define $ f(x^\pm) := \lim_{y \to x^{\pm}} f(y) $ and $ f'(x^\pm) := \lim_{y \to x^{\pm}} f'(y) $ for every $ x \in (0,1) $. To simplify the presentation, when $x\in\{0,1\}$ we set $f(x^\pm)=f(x)$ and $f'(x^\pm)=f'(x)$.
Similarly, we define $f^n(x^\pm):=\lim_{z\to x^\pm} f^n(z)$, for $n\geq 0$


\begin{defn} 
\label{de:forward}

A \textit{forward orbit} of $ x\in[0,1]$ is a sequence $\{x_n\}_{n\geq0}$ such that $x_0=x$ and either $x_n= f^n(x_0^+)$ for every $n\geq0$ or else
$x_n= f^n(x_0^-)$ for every $n\geq0$. 
An \textit{orbit segment} starting at $ x \in [0,1] $ and ending at $ y \in [0,1]$ is a finite sequence $ \{x_{0},\ldots,x_{n}\} $ with $ n>0 $ such that $ x_{0}=x $, $ x_{n}=y $ and either $ x_{k} = f^k(x_0^+) $ for every $ k=0,\ldots,n $ or else
$ x_{k} = f^k(x_0^-) $ for every $ k=0,\ldots,n $. The integer $n$ is called the \textit{length} of the orbit segment.


\end{defn}

Notice that any point $ x \in (0,1) $ has exactly two distinct forward orbits if and only if $x$ is a point of discontinuity for some $f^n$ with $n\in\Nn$. A point $x\in[0,1]$ is called \textit{regular} if $x\notin D_{f^n}$ for every $n\geq1$. Clearly, Lebesgue almost every $x\in[0,1]$ is regular.

Let $\mu$ be an ergodic acip of $f$ and define $$A_\mu:=\supp\mu.$$
\begin{defn} \label{boundary segment:def}
An orbit segment 
$\{x_0,\ldots, x_n\}$ is called a \textit{boundary segment} of $\mu$ if 
\begin{enumerate}
\item $x_0\in  D\cap \sint (A_\mu)$,
\item $x_i\in \partial A_\mu$, for all $i=1,\ldots, n-1$,
\item either $x_n=x_k$ for some $1\leq k<n$ or else $x_n\in \sint(A_\mu)$.
\end{enumerate}
\end{defn}


%


In the following lemma we show that the boundary of $A_\mu$ is determined by boundary segments.

\begin{lemma}\label{lem:boundary segment}
Every $x\in\partial A_\mu$ belongs to a boundary segment of $\mu$.
\end{lemma}

\begin{proof}
 
 We claim that every $ x \in \partial A_\mu $ is contained in an orbit segment $\{x_0,\ldots,x_p\}$ starting at a point in  $ D\cap \sint (A_\mu)$ such that $x_k\in \partial A_\mu$ for every $1\leq k \leq p$.
Indeed given $x\in\partial A_\mu$,  let
$$
E=\left\{y\in A_\mu\colon \exists\,n\in\Nn,\,  f^n(y^\pm)=x\right\}.
$$
Notice that $\sint(A_\mu)\cap E\neq\emptyset$. Indeed, suppose by contradiction that $E\subseteq \partial A_\mu$.  Denoting by $E_\delta$ a $\delta$-neighbourhood of $E$ in $A_\mu$,
since $E$ is finite  we have that $f^{-1}(E_\delta)\cap A_\mu \subseteq E_\delta$ for some  small enough $\delta>0$. Thus, $f(A_\mu\setminus E_\delta)\subseteq A_\mu\setminus E_\delta$ which contradicts the ergodicity of $\mu$. So let $y\in \sint(A_\mu)\cap E$ such that $f^n(y^{\pm})=x$ for the least possible $n\geq 1$. Then $y\in  D$, because $f(y^\pm)\in\partial A_\mu$. This proves the claim.

Consider now a forward orbit $\{z_n\}_{n\geq 0}$ of $x$ contained in $A_\mu$. Such a forward orbit always exists since $f(A_\mu\setminus D)\subseteq A_\mu$.
If $z_n\in \partial A_\mu$ for all $n\ge 1$, since $\partial A_\mu$ is finite, there exists $1\leq i<n$ such that $z_n=z_i$. Otherwise, there exists $n\geq 1$ such that $z_n\in \sint(A_\mu)$. In any case, we obtain a boundary segment containing $x$.

\end{proof}

\begin{example}
Consider the tent map $f\colon[0,1]\to[0,1]$ defined by $f(x)=2x$ if $x\leq 1/2$ else $f(x)=2-2x$.  Notice that $f$ has a unique ergodic acip $\mu$ which is the Lebesgue measure on $[0,1]$. This means that $\partial A_\mu=\{0,1\}$. The tent map has a single boundary segment $\{1/2,1,0\}$.
\end{example}


\begin{lemma}
\label{le:sing-periodic}
Suppose that $ \mu_{1} $ and $ \mu_{2} $ are two distinct ergodic acip's of $ f $. If $A_{\mu_1}\cap A_{\mu_2}\neq\emptyset$, then there exist a boundary segment of $ \mu_{1} $ and a boundary segment of $ \mu_{2} $ both ending at the same point of $ D $ or at the same periodic point.
\end{lemma}

\begin{proof} 
%

Let $ x \in \partial A_{\mu_{1}} \cap \partial A_{\mu_{2}} $, and suppose that $ f^{n}(x) \notin D $ for all $ n \ge 0 $, i.e., the forward orbit of $ x $ does not contain element of $ D $. Hence, $ f^{n+1} $ is continuous at each $ f^{n}(x) $. Since $ f(A_{\mu_{i}} \setminus D) \subset A_{\mu_{i}} $ for $ i=1,2 $, it follows that $ f^{n}(x) \in \partial A_{\mu_{1}} \cap \partial A_{\mu_{2}} $ for all $ n \ge 0 $. However, $ \partial A_{\mu_{1}} \cap \partial A_{\mu_{2}} $ is finite, and so $ x $ must be pre-periodic. By Lemma~\ref{lem:boundary segment}, the claim follows.
\end{proof}

\subsection{The separation condition}\label{sec:assumptions}

In this section we introduce a generic condition that is sufficient to prove the combinatorial stability in Section~\ref{sec stability}.

\begin{defn}\label{def:H}
We say that $f$ satisfies the \textit{separation condition} if for every ergodic acips $\mu,\nu$ of $f$ the following holds:
\begin{enumerate}
\item\label{H:1} $A_\mu\cap A_\nu=\emptyset$ whenever $\mu\neq \nu$.
\item\label{H:2} The mixing components of $\mu$ are separated, i.e.,
$$\Lambda_{i,j} \cap \Lambda_{i',j'}=\emptyset,\quad\text{whenever}\quad (i,j)\neq (i',j').$$
\item\label{H:3} $D_{f^k}\cap \partial A_\mu=\emptyset$ where $k=\Per(\mu)$.
\item\label{H:4} $(0,1)\cap \partial A_\mu $ does not contain any periodic point.
\end{enumerate}
\end{defn}

\begin{example}\label{example}
The doubling map $f(x)= 2x\pmod{1}$, $x\in[0,1]$ satisfies the separation condition. In fact, any piecewise expanding map $f$ such that $(f,\mu)$ is mixing and the support of $\mu$ is $[0,1]$, satisfies the separation condition.
\end{example}

\begin{figure}[h]
     \subfloat[A family of piecewise expanding maps where two ergodic acips collide. The middle map has an orbit segment connecting two discontinuous points.\label{subfig-2:dummy}]{%
       \includegraphics[width=0.85\textwidth]{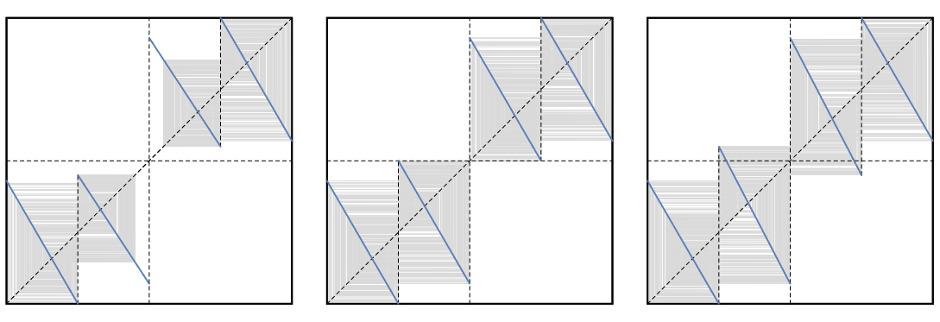}
     }
     \\
          \subfloat[The discontinuity of the Lorenz map $f_a(x)=a\,(x-1/2)\pmod{1}$ with $a=\sqrt{2}$ is pre-periodic. For $a<\sqrt{2}$ the Lorenz family has one ergodic acip of period $2$, which becomes exact for $a > \sqrt{2}$.\label{subfig-1:dummy}]{%
       \includegraphics[width=0.85\textwidth]{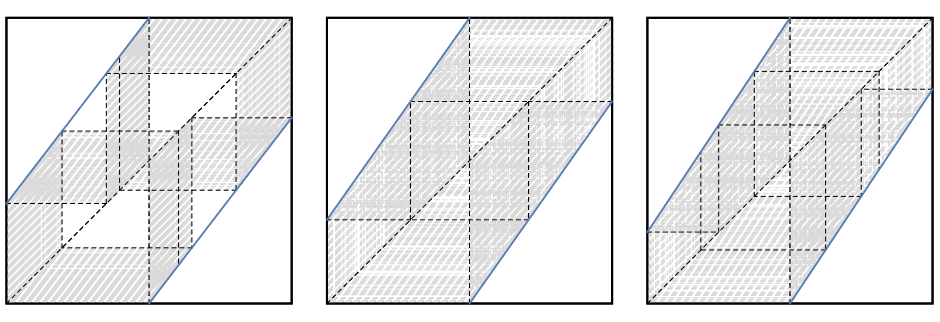}
     }
     \\
          \subfloat[A family of piecewise expanding maps where the support of the acip explodes. The middle map has a discontiniuty at a boundary point of the acip's support.\label{subfig-3:dummy}]{%
       \includegraphics[width=0.85\textwidth]{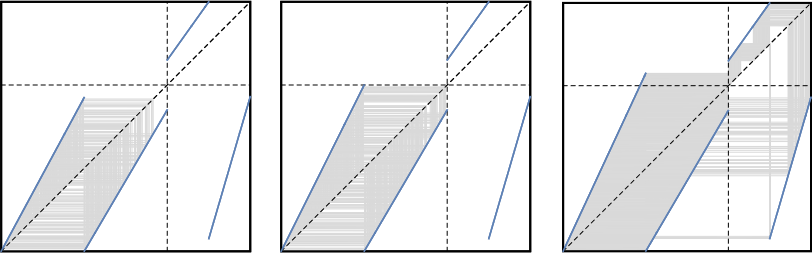}
     }
     \\
            \subfloat[A family of piecewise expanding maps where the support of the acip explodes. The middle map has a fixed point at a boundary point of the acip's support.\label{subfig-4:dummy}]{%
       \includegraphics[width=0.85\textwidth]{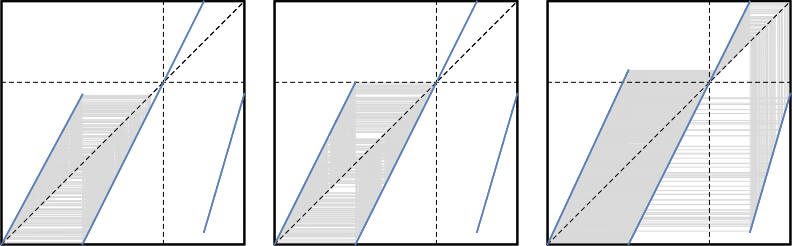}
     }
     
     \caption{Examples of families of piecewise expanding maps where the middle map does not satisfy the separation condition.}
     \label{fig}
   \end{figure}
   
See Figure~\ref{fig} for an illustration of the separation condition.
In the following we give a simpler sufficient condition that implies the separation condition.

\begin{lemma}\label{lem:sufcondition}
If there is no orbit segment starting in $ D $ and ending at a periodic point or at a point in $D$, the $f$ satisfies the separation condition.
\end{lemma}

\begin{proof}
Condition \eqref{H:1} follows directly from Lemma~\ref{le:sing-periodic}. To prove \eqref{H:2}, apply Lemma~\ref{le:sing-periodic} to the ergodic acip's $f_*^j\nu_i$ and $f_*^{j'}\nu_{i'}$ of $f^{k_{i}k_{i'}} $ where $k_i$ and $k_{i'}$ are the mixing periods of $\nu_i$ and $\nu_{i'}$, respectively. By Lemma~\ref{lem:boundary segment}, any point in $\partial A_\mu$ belongs to an orbit segment starting in $D$. So, condition \eqref{H:3} follows from the fact that no orbit segment can start and end in $D$, and condition (4) follows from the fact that no point in $D$ can be pre-periodic. Thus $f$ satisfies the separation condition.
\end{proof}

%


Next, we show that the separation condition is generic in the space of piecewise expanding maps. 

\begin{proposition}\label{prop:generic}
The set of piecewise expanding maps $f\in \XX_m$ satisfying the separation condition is residual in $\XX_m$.
\end{proposition} 

\begin{proof}

Given integers $p,n\in\Nn$ and $k\geq0$ let $\YY_{p,n,k}$ be the set of maps $f\in \XX_m$ such that there exists $x\in D_f$ with $f^{n+k}(z)=f^k(z)$ and $z:=f^p(x^\pm)$. 
Similarly, given $p\in\Nn$ let $\ZZ_{p}$ be the set of maps $f\in\XX_m$ such that there exist $x\in D_f$ and $z\in D_f$ with $f^p(x^\pm)=z$. 

The sets $\YY_{p,n,k}$ and $\ZZ_{p}$ are closed with empty interior. Hence, their union $\WW_m$ over all integers is a meagre set. Clearly, any $f\in\XX_m\setminus \WW_m$ satisfies the hypothesis of Lemma~\ref{lem:sufcondition}. Hence, the set of maps $f\in \XX_m$ satisfying the separation condition is residual.
\end{proof}

\subsection{Saturation}

Let $I\subset [0,1]$ be any open subinterval and consider the open sets $\omega_n(I)$ and $\Omega_n(I)$ defined recursively,

$$
\omega_{n+1}(I)=f(\omega_n(I)\setminus D),\quad \omega_0(I)=I,
$$
and
$$
\Omega_{n}(I)= \omega_0(I)\cup\cdots\cup\omega_{n}(I).
$$
Also define 
\begin{equation}\label{omegaI}
\Omega(I):=\bigcup_{n=0}^\infty\omega_n(I).
\end{equation}


\begin{lemma}\label{le:Omegan}
There exist $\delta>0$ and $N>0$ such that for any $n\geq N$, every connected component of $\Omega_n(I)$ has length $\geq \delta$. 
\end{lemma}

\begin{proof}
Let $ \delta(n) >0$ be the minimum length of the connected components of $ \Omega_{n}(I) \setminus  D $. Since $ D$ is finite and the sequence of open sets  $ \Omega_{n}(I) $ is increasing, $   D \cap \Omega(I) =   D \cap \Omega_{n_{0}}(I) $ for some $ n_0 \ge 0 $.  Let us show by induction that $ \delta(n) \ge \delta(n_{0}) $ for all $ n \ge n_{0} $. The statement is clearly true for $ n=n_{0} $. Suppose that the inequality is true for a given $ n \ge n_{0} $. Since $ \Omega_{n}(I) \subset \Omega_{n+1}(I) $, for each connected component $ C $ of $ \Omega_{n+1} \setminus   D $, either $ C $ contains one connected component of $ \Omega_{n}(I) \setminus   D $, or $ C $ does not intersect  $ \Omega_{n}(I) $, and in this case, it is equal to the image of the union of finitely many connected components of $ \Omega_{n}(I) \setminus   D $. In either case, the length of $ C $ is greater than or equal to $ \delta(n_{0}) $. The proof is complete.
\end{proof}

\begin{lemma}\label{le:Omega}
$\Omega(I)$ is a finite union of intervals.
\end{lemma}

\begin{proof}
Clearly, $\Omega_n(I)$ is a finite union of intervals. By Lemma~\ref{le:Omegan}, the set $\Omega_n(I)$ has a lower bound on the size of the connected components for every $n$ sufficiently large.  Hence, this implies a similar lower bound on the size of the connected components of $\Omega(I)$, thus proving the lemma.
\end{proof}

\begin{lemma}\label{le:Omeganexact}
If $(f,\mu)$ is ergodic, then $A_\mu\setminus \Omega(I)$ is a finite set for every open interval $I\subset A_\mu$.
\end{lemma}
\begin{proof}


By ergodicity, $\Omega(I)=A_\mu\pmod{0}$. Since $\Omega(I)$ is also a finite union of intervals (see Lemma~\ref{le:Omega}) the claim follows.  
\end{proof}

\subsection{Periodic orbits}

Recall that $ \Sone = \bigcup^{m}_{j} I_{j} $.
Define $ \ell(f) = \min_{j} |I_{j}| $, where $ |I_{j}| $ denotes the length of $ I_{j} $.

\begin{lemma}
\label{le:gr}
Suppose $ f $ has least expansion coefficient $ \sigma>2 $. Then for every interval $ I \subset \Sone $, there exist $ i \geq 1 $ and an open subinterval $ W \subset I $ such that 
\begin{enumerate}
\item $ f^{i}\vert_W:W\to\sint(I_{j}) $ is a diffeomorphism, for some $ 1 \le j \le m $,
\item $ f^{i+1}(W) $ is an open interval and $ | f^{i+1}(W)| \ge \sigma \ell(f) $.
\end{enumerate}
\end{lemma}

\begin{proof}
Part~(1). Let $ B = \bigcup^{m}_{j=1} \partial I_{j} $. We claim that given any interval $ I \subset \Sone $, there exist $ i \in \Nn $ and a subinterval $ W = (a,b) \subset I $ with $ a,b \in f^{-i}(B) $ such that 
\[ 
W \cap f^{-k}(B) = \emptyset \qquad \text{for } 0 \le k \le i.
\]
Indeed, if this was not the case, then for every $ i \geq 1$, no two consecutive points of $ I \cap (B \cup f^{-1}(B) \cup \cdots \cup f^{-i}(B)) $ would belong to $ f^{-i}(B) $. It is not difficult to see that this would imply that $ f^{i}(I) $ consists of at most $ 2^{i} $ intervals. But $ \sigma > 2 $, and so the length of one of these intervals would be not less than $ (\sigma/2)^{i} \to +\infty $, as $ i \to +\infty $, giving a contradiction. By the definition of $ B $, we have $ f^{i}(W) = \sint I_{j} $ for some $ j $. 

Part~(2). From Part~(1), it follows that $ f^{i+1}(W) = f(\sint I_{j}) $ is an open interval, and so
\[ 
|f^{i+1}(W)| = |f(I_{j})| \ge \sigma |I_{j}| \ge \sigma \ell(f).
\]
\end{proof}

%

Let $\mu$ be an ergodic acip of $f$. In the next theorem we prove, using elementary methods, that the periodic points of $ f $ are dense in the support of $\mu$. 

\begin{theorem}
\label{pr:density}
The periodic points of $ f $ are dense in $A_\mu$. 
\end{theorem}

\begin{proof}

To obtain the wanted conclusion, we show that every open interval $ U \subset A_\mu $ contains a periodic point of $ f $.
 Let $k\in\Nn$ such that $f^k$ has least expansion coefficient $>2$. 
Also, let $\mathcal{I}_\mu$ be the collection of the connected components of $\sint(A_\mu\setminus   D_{f^k})$. 

By Lemma~\ref{le:Omeganexact}, we can assume that $U\subset \Omega_N(I)$ for some large enough integer $N$
and all $I\in \mathcal{I}_\mu$. Then reducing the open interval  $U$ even further we can
assume that for every $I\in \mathcal{I}_\mu$ there exists a positive integer $n_I\leq N$
such that $U\subset \omega_{n_I}(I)$.
We conclude that $f^{n_I}\vert_{I'} \colon I'\to U$ is a diffeomorphism for some open subinterval $I'\subset I$.

By Lemma~\ref{le:gr}, there exists an open subinterval $ W \subset U $ and $i\geq 1$ such that $ g:=f^{ik}\vert_W:W\to I $ is a diffeomorphism for some $I\in\mathcal{I}_\mu$.

Let $W':=g^{-1}(I')$. Then $f^{ik+n_I}\vert_{W'}\colon W'\to U$ is a diffeomorphism and $W'\subset W\subset  U$. Since $f^{ik+n_I}$ is expanding, it has a fixed point inside $U$. This proves the theorem. 
\end{proof}

In the next result we give a characterization for $(f,\mu)$ to be exact in terms of periodic orbits. 

\begin{proposition}\label{prop:exact}
$(f,\mu)$ is exact if and only if any open set in $A_\mu$ contains two periodic points having coprime periods.
\end{proposition}

\begin{proof}
If $(f,\mu)$ is not exact, then it has $k\geq 2$ mixing components. Therefore, there is an open set $U$ in $A_\mu$ where $k$ must divide the period of any periodic point in $U$. This shows that the coprimality condition of the periods is sufficient for exactness. To show that it is necessary, suppose that $(f,\mu)$ is exact. By exactness, $\mu(f^n(I))\to 1$ as $n\to\infty$ for any interval $I\subset [0,1]$. Let $U$ be an open subinterval of $A_\mu$. Shrinking $U$ if necessary, for every $I\in \mathcal{I}_\mu$ there exists a positive integer $n_I\geq 1$
such that $U\subset \omega_{n_I}(I)$ and $U\subset \omega_{n_I+1}(I)$. Here, $\mathcal{I}_\mu$ denotes the collection of the connected components of $\sint(A_\mu\setminus   D)$. Arguing as in the proof of Theorem~\ref{pr:density}, we conclude that both $f^{i+n_I}$ and $f^{i+n_I+1}$ have a fixed point inside $U$ for some $I\in\mathcal{I}_\mu$ and integer $i\geq 1$. Since $U$ is arbitrary, this shows the existence of two periodic orbits with coprime periods in any open set in $A_\mu$.
\end{proof}

\begin{remark}
The existence of a fixed point in $A_\mu$ is not sufficient to show that $(f,\mu)$ is exact. Indeed, consider the orientation-reversing Lorenz map $f(x)=1-a(x-1/2)\pmod{1}$ with $a=\sqrt{2}$. 
\end{remark}

\begin{remark}\label{rem:exact}
Let $\mu$ be an ergodic acip with separated mixing components (see Part~\eqref{H:2} of the separation condition). If $A_\mu$ contains two periodic points with coprime periods, then $(f,\mu)$ is exact.
\end{remark}

\section{Combinatorial Stability}\label{sec stability}
\label{sec:ergodic stability}
In this section, we prove that a piecewise expanding map $f$ is combinatorially stable provided it satisfies the separation condition. Recall by Proposition~\ref{prop:generic}, the separation condition is generic in the space of piecewise expanding maps on $m$ intervals. 

Let $\SRB(f)$ denote the finite set of ergodic acip's of $f$. 

%


\begin{theorem}
\label{thm:main2}
If $f\in\XX_m$ satisfies the separation condition, then there is a neighbourhood $\VV$ of $f$ in $\XX_m$ such that for every  $g\in\VV$, there is a bijection $\Theta_g$ between $\SRB(f)$ and $\SRB(g)$ satisfying $\Theta_f=\id$ and $\Per(\Theta_g(\mu))=\Per(\mu)$ for every $ \mu \in \SRB(f) $ and $g\in\VV$. Furthermore, for every $\mu\in \SRB(f)$, the map $\VV\ni g\mapsto A_{\Theta_g(\mu)}$ is continuous at $f$ with respect to the Hausdorff metric.

\end{theorem}
Under the generic separation condition, Theorem~\ref{thm:main2} shows that the number of ergodic acip's of
$ f $ and the corresponding mixing periods do not change in a neighbourhood of $f$, a property we call \textit{combinatorial stability}. 

\begin{remark}

The separation condition does not prevent the attractor of a perturbation of $f$ from creating a `hole', i.e., split the mixing components without changing its period. An ingredient to create such a hole is the existence of two distinct orbit segments starting in $ D$ and ending at the same point. See Figure~\ref{fig:C} below.

\begin{figure}[h]
\begin{center}
\includegraphics[scale=.35]{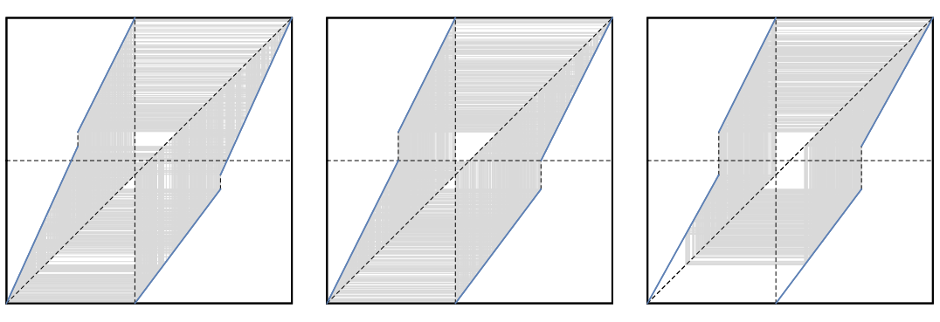}
\end{center}
\caption{Example of a family of piecewise expanding maps where a `hole' appears inside an  exact ergodic acip. The middle map is combinatorially stable since satisfies the separation condition, but has two distinct discontinuities $d_1$ and $d_2$ such that $f(d_1^-)=f(d_2^+)$.}
\label{fig:C}
\end{figure}

\end{remark}


\begin{remark}
If $(f,\mu)$ is mixing and $\mu$ is supported on $[0,1]$, then $f$ is combinatorially stable, i.e., any small perturbation of $f$ has a unique ergodic acip which is also mixing (see Example~\ref{example}).
\end{remark}

Theorem~\ref{thm:main2} is proved in Section~\ref{sec:bijection theta}. In the following we prove several preliminary lemmas.

\subsection{Perturbation lemmas}

The following lemma follows from standard considerations in hyperbolic theory. For the convenience of the reader we include here a proof. 

\begin{lemma}\label{lem:periodic points}
Let $x\in[0,1]$ be a regular periodic point of $f$. There is a neighbourhood $\VV$ of $f$ and a continuous map $g\mapsto x_g$ defined on $\VV$ such that $x_g$ is a periodic point of $g$ having the same period of $x$.
\end{lemma}

\begin{proof}
Let $k>0$ be the period of $x$. Since $f$ is expanding there is an interval $J$ containing $x$ and with closure not intersecting $ D$ such that $f^k(J)$ is an interval, $J\subset f^k(J)$ and $h_f:=f^k|_J$ is a bijection. By continuity, there is a neighbourhood $\VV$ of $f$ such that the same holds for every $g\in\VV$, in particular $h_g:=g^k|_J$ is a bijection. Consider the map $\varphi\colon \VV\times J\to J$ defined by $\varphi(g,x)=h_g^{-1}(x)$. Clearly, $\varphi$ is continuous and $\varphi(g,\cdot)$ is a uniform contraction. Therefore, by the contraction fixed point theorem (continuous dependence of parameters version), $\varphi(g,\cdot)$ has a unique fixed point $x_g$ which depends continuously on $g\in\VV$.
\end{proof}

We say that \textit{$x_g$ is the continuation of $x$ by $g$}.
\medskip

Let $\mu$ be an ergodic acip of $f$. 
In the following lemma we use the boundary segments associated to $\partial A_\mu$ to define a closed forward invariant set for small perturbations of $f$. 


\begin{lemma}\label{lem:trapping}
If $D\cap \partial A_\mu=\emptyset$ and $(0,1)\cap\partial A_\mu$ does not contain periodic points, then there exist a neighbourhood $\VV$ of $f$ and a continuous\footnote{In the Hausdorff metric.} map $g\mapsto U_\mu(g)$ defined on $\VV$ such that
\begin{enumerate}
\item $U_\mu(f)=A_\mu$,
\item $U_\mu(g)$ is a finite union of closed intervals for every $g\in\VV$,
\item $g(U_\mu(g))\subseteq U_\mu(g)$ for every $g\in\VV$.
\end{enumerate}
\end{lemma}

\begin{proof}
Notice that, for every $x\in D\cap \sint (A_\mu)$ we have $f(x^\pm)\in A_\mu$. By slightly abusing notation we shall call the points in $ D\cap \sint (A_\mu)$ together with their images that satisfy $f(x^\pm)\in\sint(A_\mu)$ also boundary segments.
Consider the collection $\BB$ of all boundary segments of $\mu$. Clearly, given $d\in D\cap \sint (A_\mu)$ there is $\gamma=\{x_0,\ldots,x_n\}\in\BB$ such that $x_0=d$. Moreover, by Lemma~\ref{lem:boundary segment}, every point in $\partial A_\mu$ is contained in a boundary segment in $\BB$. Notice that, two or more points in $\partial A_\mu$ may be covered by a single boundary segment and a single point in $\partial A_\mu$ may be covered by more than one boundary segment.

By the hypothesis $D\cap \partial A_\mu=\emptyset$, we have $f(A_\mu)\subseteq A_\mu$. This together with the hypothesis $(0,1)\cap\partial A_\mu$ has no periodic points, implies that given any boundary segment $\gamma=\{x_0,\ldots,x_n\}\in\BB$ it satisfies the properties:
\begin{enumerate}
\item $x_1=f(x_0^\pm)$ and $x_0\in D\cap\sint(A_\mu)$,
\item $x_{i+1}=f(x_i)$ for each $i=1,\ldots,n-1$,
\item $x_i\in\partial A_\mu$ and $x_i\notin D$ for each $i=1,\ldots,n-1$,
\item one of the following alternatives hold:
\begin{enumerate}
\item $x_n\in\sint(A_\mu)$,
\item $x_n=x_{n-2}\in\{0,1\}$,
\item $x_n=x_{n-1}\in\{0,1\}$.
\end{enumerate}
\end{enumerate}

For each $x\in \partial A_\mu$ let $\BB_x\subset \BB$ denote the collection of all boundary segments passing through $x$. By previous observations, $\BB_x\neq\emptyset$. So we may define the \textit{order} of $x\in\partial A_\mu$ to be
$$
\text{ord}(x):=\max\{k\in\Nn\colon x_k=x\text{ for some }\{x_0,\ldots,x_n\}\in\BB_x\}.
$$
It is convenient to set $\text{ord}(x)=0$ for any $x\in D\cap\sint(A_\mu)$. Notice that, $\ord(z)\leq \ord(x)$ whenever $x=f(z^\pm)$ with $x,z\in \partial A_\mu\cup(D\cap\sint(A_\mu))$.

Define $E:=\partial A_\mu\cup(D\cap\sint(A_\mu))$. The points in $E$ induce a partition of $A_\mu$ as a union of closed intervals,
$$
A_\mu=[\alpha_{0},\beta_0]\cup[\alpha_1,\beta_1]\cdots\cup [\alpha_{q},\beta_q],
$$
where $\alpha_{0}<\beta_0\leq\alpha_1<\beta_1\leq\cdots\leq \alpha_q<\beta_q$ and $\alpha_i,\beta_i\in E$. Notice that two consecutive intervals are either disjoint or intersect at a single point belonging to $ D\cap\sint(A_\mu)$. 

Now, let $g\in \VV$ where $\VV$ is an $\varepsilon$-neighbourhood of $f$ and $\varepsilon>0$ is sufficiently small. We define a map $\varphi_g\colon E\to [0,1]$ in the following way. Given $x\in D\cap\sint(A_\mu)$, we set $\varphi_g(x):=a_i(g)$ where $x=a_i(f)$ for some $1\leq i\leq m-1$. Otherwise, suppose that $x\in\partial A_\mu$. We have two cases. Either $x$ is periodic under $f$ or $x$ is not periodic. In the first case, $x\in\{0,1\}$, and we set $\varphi_g(x)=x$. In the second case, we will define $\varphi_g$ on $E$ inductively on the order of the points. So suppose that $\varphi_g$ has been defined for points in $E$ whose order is $\leq n$. Let $x\in E$ such that $\text{ord}(x)=n+1$. We suppose that $x$ is a left boundary point of $\partial A_\mu$, i.e., $[x,x+\delta]\subset A_\mu$ for $\delta>0$ small. The case of $x$ being a right boundary of $A_\mu$ is treated similarly (the $\min$ below becomes a $\max$). Then we define
$$
\varphi_g(x):=\min\{g(\varphi_g(z)^\pm)\colon z\in E,\, x=f(z^\pm)\}.
$$
In this way we have a well-defined map $\varphi_g:E\to[0,1]$. Choosing $\varepsilon>0$ smaller, if necessary, $\varphi_g$ becomes injective. Moreover, $\varphi_f=\id$. Using the map $\varphi_g$ we finally define,
$$
U_\mu(g):=\left[\varphi_g(\alpha_0),\varphi_g(\beta_0)\right]\cup\cdots\cup\left[\varphi_g(\alpha_{q}),\varphi_g(\beta_q)\right].
$$
Now it is simple to check that $U_\mu(g)$ satisfies all properties stated in the lemma. 
\end{proof}

Recall that $ \Sone = \bigcup^{m}_{j} I_{j} $
and $ \ell(f) = \min_{j} |I_{j}| $, where $ |I_{j}| $ denotes the length of $ I_{j} $.

\begin{lemma}
\label{co:ep}
There is a neighbourhood $\VV$ of $f$ and a constant $ \eta=\eta(\VV)>0 $ such that for every interval $ I \subset \Sone $ there exists $ n \geq 1 $ for which $ g^{n}(I) $ contains an open interval of length greater than $ 2\eta $ for every $g\in\VV$.
\end{lemma}

\begin{proof}
Apply Lemma~\ref{le:gr} to $ f^{k} $ with $ k > 0 $ being the smallest integer such that the least expansion of $ f^{k} $ is greater than 2. Then $ n = i k $, where $ i $ is as in Lemma~\ref{le:gr} (applied to $ f^{k} $), and $ \eta(f) = \ell(f^{k}) $. 

Let $\eta:=\min_{g\in\VV}\eta(g)$. Since, for every $g$ sufficiently close to $f$, the integer $n$ can be made uniform (not depending on $g$), the conclusion of the lemma follows. 
\end{proof}

\begin{defn}
Given points $x,y\in [0,1]$
we say that {\em $x$ leads to $y$
under  $f$}, and write $x\leadsto y$,
if for every neighbourhood $V$ of $x$ there exists
$n\geq 0$ such that $y\in f^n(V)$. We say that two points are {\em heteroclinically related under  $f$} if $x\leadsto y$ and $y\leadsto x$. 
\end{defn}

Clearly, the heteroclinic relation is an equivalence relation. 
Another key observation is that the heteroclinic relation between periodic points is stable under perturbation.

\begin{lemma}\label{lem:stable hetero}
If two regular periodic points $x$ and $y$ of $f$
are heteroclinically related under  $f$, then there is a neighbourhood $\VV$ of $f$ such that for every $g\in \VV$ the continuations $x_g$ and $y_g$ of the 
periodic points $x$ and $y$ are also heteroclinically related under $g$.
\end{lemma}

\begin{proof}
Let $x$ and $y$ be two regular periodic points for $f$ such that  $x\leadsto y$.
By Lemma~\ref{lem:periodic points}, $x$ and $y$ have continuations $x_g$ and $y_g$ for every $g\in\VV$ where $\VV$ is a neighbourhood of $f$. We will show that $x_g\leadsto y_g$. Denote by $p$ the period of $x$ and $x_g$.
Define
$$
\tau:= \frac12\inf_{g\in\VV} \dist(x_g,D_{g^p})>0.
$$  
Notice that $I_\tau(x_g)\cap D_{g^p}=\emptyset$ for every $g\in\VV$ where $I_\tau(z):=(z-\tau,z+\tau)$.
Since $x\leadsto y$, there is $n=n(x,\tau)\geq0$ such that $y\in f^n(I_\tau(x))$. Shrinking $\VV$ if necessary, we may assume that $y_g\in g^n(I_{\tau}(x_g))$ for every $g\in\VV$.
Now let $J$ be any interval containing $x_g$. Since $g$ is expanding, there is $k\geq0$ such that $I_\tau(x_g)\subset g^{kp}(J)$. Thus, $y_g\in g^{kp+n}(J)$. This shows that $x_g\leadsto y_g$.
\end{proof}


\subsection{Proof of Theorem~\ref{thm:main2}}
\label{sec:bijection theta}

Let $\mu\in\SRB(f)$ be an ergodic acip of a piecewise expanding map $f$ satisfying the separation condition. 


We divide the proof of Theorem~\ref{thm:main2} in four lemmas. Throughout the proof, we  assume that the neighbourhood $\VV$ of $f$ is chosen to be sufficiently small so that the hypothesis of the perturbation lemmas are verified.

\begin{lemma}\label{lem:injectivity}
There is a neighbourhood $\VV$ of $f$ such that for each $g\in\VV$, there exists a unique $\nu\in \SRB(g)$ such that $A_\nu\subseteq U_\mu(g)$.
\end{lemma}

\begin{proof}

Let $\VV$ be a neighbourhood of $f$ for which the conclusion of Lemma~\ref{lem:trapping} holds. Given $g\in\VV$, the existence of $\nu\in\SRB(g)$ such that $A_\nu\subseteq U_\mu(g)$ follows directly from Part (3) of Lemma~\ref{lem:trapping} and Theorem~\ref{th:acip}. 

To prove the uniqueness, suppose that $\nu_1$ and $\nu_2$ are two ergodic acips of $g\in\VV$ whose supports are contained in $U_\mu(g)$.
We want to show that $\nu_1=\nu_2$.

Take $\eta= \eta(\VV)>0$ given by Lemma~\ref{co:ep}. Let $k\in\Nn$ such that $f^k$ has least expansion coefficient $>2$.
Also let $\II_\mu$ be the set of connected components of $\sint(A_\mu\setminus D_{f^k})$ and $\Omega$ the intersection of $\Omega(I)$ over all $I$ belonging to $\II_\mu$. Recall that $\Omega(I)$ is defined in \eqref{omegaI}. By Lemma~\ref{le:Omeganexact}, $\Omega$ equals $A_\mu$ except for a finite set of points, which we denote by $E$. Since periodic points are dense in $A_\mu$ (by Theorem~\ref{pr:density}), we can take a $\eta/3$-dense set $X:=\{x_1,\ldots, x_r\}\subset \Omega\setminus D$
of regular periodic points of $f$ such that $x_i$ and $x_j$ are heteroclinically related under  $f$ for all $1\leq i,j\leq r$.
Indeed, by Lemma~\ref{le:gr}, for any $x\in X$ there is $I\in \II_\mu$ such that $I\subset f^{nk}(V)$ for some $n\in\Nn$ and neighbourhood $V$ of $x$. But $\Omega(I)$ contains $A_\mu\setminus E$. Therefore, $x\leadsto y$ for any $y\in X$.

Let $X':=\{x_1',\ldots, x_r'\}$ denote the set of continuations of the periodic points in $X$ for some nearby map $g\in\VV$. According to Lemma~\ref{lem:stable hetero}, by choosing $\VV$ sufficiently small, we can assume that 
$x_i'$ and $x_j'$ are heteroclinically related under $g$ for all $1\leq i,j\leq r$. We can also assume that $\abs{x_j-x_j'}<\eta/3$
for all $j=1,\ldots, r$ and that the Hausdorff distance between $U_\mu(f)$ and $U_\mu(g)$ is also less than $\eta/3$. This follows from the continuity of the maps in Lemma~\ref{lem:periodic points} and Lemma~\ref{lem:trapping}.

Take now  points $y_i\in A_{\nu_i}\subseteq U_\mu(g)$ and   neighbourhoods $V_i$ of $y_i$ in $A_{\nu_i}$. By Lemma~\ref{co:ep}, there exist subintervals $I_i\subset V_i$,  integers $n_i\geq 1$ and points $z_i\in U_\mu(g)$ such that $g^{n_i}(I_i) =(z_i-\eta,z_i+\eta)$. Because of the previous considerations, 
\begin{align*}
\text{dist}(z_i,X')&\leq \text{dist}(U_\mu(g),U_\mu(f))+\text{dist}(U_\mu(f),X)+\text{dist}(X,X')\\
&< \frac\eta3+\frac\eta3+\frac\eta3= \eta.
\end{align*}
Thus, $g^{n_i}(I_i)\cap X'\neq \emptyset$ which implies that $y_i$ leads to
a periodic point in $X'$ under $g$.
Finally, since the points $y_i\in A_{\nu_i}$ are arbitrary,
and all points in $X'$ are heteroclinically related under $g$, it follows that $\nu_1=\nu_2$.
\end{proof}

Lemma~\ref{lem:injectivity} shows that we have a well-defined map for every $g\in\VV$,
$$
\Theta_g\colon\SRB(f)\to \SRB(g),
\qquad
\mu\mapsto \nu,
$$ 
where $\nu\in\SRB(g)$ comes from Lemma~\ref{lem:injectivity}. Clearly, $\Theta_f=\id$. In the following two lemmas we show that $\Theta_g$ is a bijection as stated in Theorem~\ref{thm:main2}. 

\begin{lemma}
The map $\Theta_g$ is one-to-one.
\end{lemma}

\begin{proof}
By Part~\ref{H:1} of the separation condition, $A_{\mu_1}\cap A_{\mu_2}=\emptyset$ for every $\mu_1\neq\mu_2$ in $\SRB(f)$. 
We can further suppose that $\VV$ is sufficiently small so that $U_{\mu_1}(g)\cap U_{\mu_2}(g)=\emptyset$ for every $g\in\VV$ and every $\mu_1\neq\mu_2$ in $\SRB(f)$. Now suppose that $\nu=\Theta_g(\mu_1)=\Theta_g(\mu_2)$.
But, $A_{\nu}\subseteq U_{\mu_i}(g)$ which can only happen if $\mu_1=\mu_2$. 
\end{proof}

\begin{lemma}
\label{le:onto}
The map $\Theta_g$ is onto.
\end{lemma}

\begin{proof}
By Part~(3) of Theorem~\ref{th:acip}, the union $B$ of all basins of attraction $B(\mu_i)$\footnote{Recall that $x\in B(\mu)$ iff for every continuous function $\varphi:[0,1]\to\Rr$ we have $$\lim_{n\to\infty}\frac{1}{n}\sum_{k=0}^{n-1}\varphi (f^n(x))=\int\varphi\,d\mu.$$} over the elements $\mu_i\in\SRB(f)$ coincides with the interval $[0,1]$ up to a zero Lebesgue measure set. Let $\eta=\eta(\VV)>0$ be the constant in Lemma~\ref{co:ep}. 

Given $\mu\in\SRB(f)$, let $\varphi_\mu$ be a continuous function having compact support inside $\sint(A_\mu)$. It follows that
 $$\lim_{n\to\infty}\frac1n \sum^{n-1}_{k=0} \varphi(f^{k}(x)) =\int\varphi_\mu\,d\mu > 0,\quad\forall\, x \in B(\mu).
 $$
Hence, for every $x\in B(\mu)$ there are infinitely many integers $k_i\geq0$ such that $f^{k_i}(x)\in\sint(A_\mu)$. This implies that there is an $\eta/2$-dense set $Z:=\{z_1,\dots, z_r\} \subset B$ of $[0,1]$ such that for every $z_i\in Z$, we can find $k \in\Nn$ and $\mu \in \SRB(f)$ for which $f^{k}(z_{i}) \in \sint (A_\mu )$.

Now, let $\mu'\in\SRB(g)$ and take $y\in  A_{\mu'}$. Also let $W$ be a neighbourhood of $y$ in $A_{\mu'}$. By Lemma~\ref{co:ep}, there is $ n \in \Nn $ such that $g^n(W)$ contains an interval of length greater than or equal to $ \eta$. Thus, $g^n(W)$ contains in its interior a point $z_i\in Z$, i.e., $y$ leads to $z_i$ under $g$. It follows that the intersection $g^{n+t_{i}}(W) \cap U_{\mu_{j_i}}(g) $ contains an interval for some $t_i\geq0$ and $1\leq j_i\leq \#\SRB(f)$. Arguing as in the proof of Lemma~\ref{lem:injectivity},  this shows that $\mu'=\Theta_g(\mu_{j_i})$, and thus $\Theta_g$ is onto.

\end{proof}

It remains to show that $f$ and $g\in\VV$ have the same number of mixing components.

\begin{lemma}
$\Per(\Theta_g(\mu))=\Per(\mu)$ for every $ \mu \in \SRB(f) $ and $g\in\VV$.
\end{lemma} 

\begin{proof}

Consider an ergodic acip $\mu$ for $f$. Let $k:=\Per(\mu)$. We can write,
$$
U_\mu(g)=U^{(1)}_\mu(g)\cup\cdots\cup U^{(k)}_\mu(g)
$$
where $U^{(i)}_\mu(g):=U_{f^i_{*}\nu}(g^k)$ are the sets as in Lemma~\ref{lem:trapping} with $(f,\mu)$ replaced by the exact piecewise expanding map $(f^k,f_{*}^i\nu)$ (see Theorem~\ref{th:acip}).  Since $f$ satisfies Part~\ref{H:2} of the separation condition, the sets $U^{(i)}_\mu(g)$ are pairwise disjoint for every $g\in\VV$.

Let $\mu':=\Theta_g(\mu)$ be the unique ergodic acip for $g\in \VV$
such that $A_{\mu'}\subseteq U_\mu(g)$ and define $k':=\Per(\mu')$. 
We first notice that $k$ divides $k'$ because 
$$g^k(U^{(j)}_\mu(g))\subseteq  U^{(j)}_\mu(g),\quad\forall\, j=0,\ldots, k-1.$$ 
To prove that $k=k'$ we will assume without loss of generality that $k=1$.
For the general case we can replace $g$ by $g^k$,
resp. $f$ by $f^k$.
So we suppose that $f$ has a unique mixing component $A_{\mu}$, i.e., $(f,\mu)$ is exact. 

Let $\eta=\eta(\VV)>0$ be the constant in Lemma~\ref{co:ep} and $X\subset A_\mu$ be a finite set of
regular periodic points of $f$ with the property that every sub-interval $J\subset A_\mu$
of length greater or equal than $ \eta/2$ contains at least two periodic points in $X$ with coprime periods. 
This is possible by Proposition~\ref{prop:exact}. 

Shrinking the neighbourhood $\VV$ if necessary, we may assume that $A_\mu\cap A_{\mu'}$ contains an interval $I$ whose length is $\geq \eta/2$. Thus, $I$ contains two periodic points $x$ and $y$ in $X$ with coprime periods. According to Lemma~\ref{lem:periodic points}, these periodic points have continuations $x_g,y_g\in I$ for every $g\in\VV$ whose periods are coprime as well. Thus, by Remark~\ref{rem:exact}, we conclude that $(g,\mu')$ is exact.

\end{proof}

Finally, to complete the proof of Theorem~\ref{thm:main2}, it remains to prove that the map $g\ni \VV\mapsto A_{\Theta_g(\mu)}$ is continuous at $f$. By Lemma~\ref{lem:trapping} and because $A_\mu=U_\mu(f)$, the map $g\ni \VV\mapsto A_{\Theta_g(\mu)}$ is upper semi-continuous at $f$. The lower semi-continuity at $f$ follows from the density of periodic points (Theorem~\ref{pr:density}) and the fact that any finite set of heteroclinically related regular periodic points is stable (Lemma~\ref{lem:stable hetero}). 

\section*{Acknowledgements}

The authors were partially supported by Funda\c c\~ao para a Ci\^encia e a Tecnologia, GDM, JLD and JPG through the strategic project PEst-OE/EGE/UI0491/2013, PD through the strategic project PEst-OE/MAT/UI0209/2013, and
JPG was also supported through the grant SFRH/BPD/78230/2011.


\end{document}